\documentclass[11pt]{article}

\usepackage{amsmath,amssymb,amsthm,setspace,graphicx}
\usepackage[text={6.5in,8.4in}, top=1.3in]{geometry}
\usepackage{hyperref}

\newtheorem{theorem}{Theorem}[section]
\newtheorem{proposition}[theorem]{Proposition}
\newtheorem{lemma}[theorem]{Lemma}

\newtheorem{conjecture}[theorem]{Conjecture}

\begin{document}

\title{On Ramsey Properties of  $k$-Majority Tournaments}

\author{
Asaf Shapira
\thanks{School of Mathematics, Tel Aviv University, Tel Aviv 6997801, Israel. Email: asafico@tau.ac.il\,.}
\and
Raphael Yuster
\thanks{Department of Mathematics, University of Haifa, Haifa 3103301, Israel. Email: raphael.yuster@gmail.com\,.}
}

\date{}

\maketitle

\setcounter{page}{1}

\begin{abstract}

A central objective in Ramsey theory is determining whether restricted families of discrete structures necessarily contain substantially larger homogeneous substructures, compared to the unrestricted structures. In the setting of tournaments, it is well known that every tournament contains a transitive subgraph of size $\log n$, and that this is best possible up to a constant factor. A restricted family of tournaments that has been extensively studied is
the family of $k$-majority tournaments. They are obtained by taking $2k-1$ linear orders of a set $X$, and defining a tournament on $X$
which has an edge from $u$ to $v$ if $u$ precedes $v$ in at least $k$ of these orders.
Milans, Schreiber, and West proved that such tournaments indeed have significantly larger transitive tournaments.
More precisely, they proved that every $k$-majority tournament contains a transitive tournament of size
$n^{2^{-\Theta(k)}}$. Our main goal in this paper is to give an exponential improvement in the dependence of the exponent on $k$ by showing
that every $k$-majority tournament contains a transitive set of size $n^{\Omega(1/k)}$.
Finally, we highlight several open problems and conjectural directions related to random $k$-majority tournaments.
\end{abstract}

\section{Introduction}

A set of vertices in a graph is homogeneous if it is either complete or empty.
Ramsey's theorem states that for every $n$ there is an integer $N=r(n)$ so that every
graph contains a homogeneous set of size $n$. Obtaining tight upper and lower bounds
for $r(n)$ is one of the central problems in extremal and probabilistic combinatorics, with some spectacular recent progress.
We refer the reader to \cite{morris-2026} for a recent survey on the topic. It is well known that $r(n)$ is exponential in $n$,
that is, that in the worst case, $n$-vertex graphs contain homogeneous sets of size $O(\log n)$.
It has long been observed that graphs belonging to many restricted families of graphs contain much larger homogeneous sets.
For example, a recent theorem of Nguyen, Scott and Seymour \cite{NSS-2025} states that graphs of bounded VC dimension contain
homogeneous sets of size $n^{\Theta(1)}$. A central conjecture in this area due to Erd\H{o}s and Hajnal \cite{EH-1989} further states
that for every fixed $H$, such a statement holds in induced $H$-free graphs.

In this paper we study tournaments, that is, directed graphs obtained by orienting the edges of the complete graph.
Let $T_t$ denote the transitive tournament on $t$ vertices.
In the setting of tournaments it is natural to say that a set of $t$ vertices is homogeneous if it induces a copy of $T_t$.
It was already observed by Erd\H{o}s and Moser \cite{EM-1964} that every $n$-vertex tournament contains a homogeneous set of size
$\log n$ and that this is best possible up to a factor of $2$. It is again natural to ask if much larger homogeneous sets can be guaranteed to exist in tournaments belonging to certain restricted families of graphs. We next recall the definition of one such family.

For $k \ge 1$ and a set $\Pi$ of $2k-1$ linear orders of a set $X$, the {\em $k$-majority tournament
generated by $\Pi$} is the tournament $G$ with $V(G)=X$ and $(u,v) \in E(G)$ if $u$ precedes $v$ in at least $k$ of the orders. Such tournaments were first studied by McGarvey \cite{mcgarvey-1953}
as a tool for analyzing voting paradoxes; he noticed that every tournament is a $k$-majority tournament for some $k$. Since then, $k$-majority tournaments have been further studied by several researchers
as these seem to have certain well-behaved properties that turn out to be far from trivial already for $k=2$. Most notably, we mention that Erd\H{o}s  and Moser \cite{EM-1964} proved that every tournament with
$\Theta(k \log k)$ vertices is a $k$-majority tournament and Stearns \cite{stearns-1959} proved that
the $k \log k$ growth rate is best possible. An interesting nontrivial property
of $k$-majority tournaments is that their domination number is bounded; indeed, Alon et al. \cite{ABKKW-2006} proved that their domination number is $O(k \log k)$. In particular, even if the elements of $\Pi$ are generated at random, the resulting $k$-majority tournament is far from a typical random tournament.


Returning to our Ramsey theoretic discussion above, it is now natural to ask, for a given integer $k$, what is the largest
$t=f_k(n)$ so that every $n$-vertex $k$-majority tournament contains a copy of $T_t$.
The function $f_k(n)$ was first studied by Milans, Schreiber, and West \cite{MSW-2011}, showing that
$f_k(n)$ is polynomial in $n$. More precisely, they showed that
$n^{1/c_k} \le f_k(n) \le n^{1/d_k}$ where $c_k \le 3^{k-1}$ and $d_k \ge (1+o_k(1))\frac{\ln k}{\ln\ln k}$. Notice, however, that the gap in the exponents is doubly exponential in $k$. The situation is
more satisfactory for $f_2(n)$ as they have proved that $f_2(n)=\Theta(\sqrt{n})$.
Already for $k=3$ the growth rate of $f_k(n)$ is not known. Notice also that it is not entirely obvious that
$\log_n f_k(n)$ has a limit. Our first main result significantly improves the lower bound
for $f_k(n)$
establishing that $c_k \le 2k-\frac{1}{2}\log_2 k$. Furthermore, we prove that the limit above exists.
Together with the $d_k$ upper bound from above we obtain the following theorem.

\begin{theorem}\label{t:1}
	$\lim_{n \rightarrow \infty} \log_n f_k(n)$ exists. Furthermore, if $1/c_k$ is this limit, then
	$$
	(1+o_k(1))\frac{\ln k}{\ln\ln k} \le c_k \le 2k-\frac{1}{2}\log_2 k\;.
	$$
\end{theorem}

So rephrasing the above theorem, we see that while \cite{MSW-2011} showed that every $k$-majority tournament contains
a $T_t$ with $t=n^{2^{-\Theta(k)}}$, the theorem above exponentially improves the exponent's dependence on $k$, by
guaranteeing a $T_t$ with $t=n^{\Omega(1/k)}$.

\subsection{A bipartite variant of \texorpdfstring{$f_k(n)$}{fk(n)}}

We now turn to discuss the main technical result in this paper which deals with a bipartite variant of the above
problem, which might be of independent interest. Let $T_{t,t}$ denote the transitive bipartite tournament with $t$ vertices in each part,
that is, an orientation of the complete bipartite graph $K_{t,t}$ where all edges point from one of the vertex sets to the other.
Note that finding copies of a large $T_{t,t}$ in a tournament can be thought of as the tournament analogue of
the classical Zarankiewicz’s problem in undirected graphs.
While random tournaments show that there are tournaments which do not contain
a $K_{t,t}$ with $t > \log n$, it is again natural to ask if $k$-majority tournaments contain much larger $K_{t,t}$s.
Let us then set $b_k(n)$ to be the largest integer $t$, such that every
$n$-vertex $k$-majority tournament contains a $T_{t,t}$.

A standard inductive argument shows that a lower bound for $f_k(n)$ can be obtained from a lower bound
for $b_k(n)$ (see Section 3), hence this is further motivation to seek a reasonable lower bound for the latter.
As it turns out, to obtain such a lower bound, it
is worthwhile to consider somewhat stronger variants of transitive bipartitions which are meaningful in $k$-majority tournaments and which are defined as follows. Let $\pi_1,\ldots,\pi_{2k-1}$ be linear orders of an $n$-set,
and let $G$ denote the corresponding $k$-majority tournament.
For two disjoint sets of vertices $A$ and $B$ of $V(G)$, we say that
{\em $A$ dominates $B$ in $\pi_i$} if all the vertices of $A$ appear in $\pi_i$ before all the vertices of $B$. We say that $A$ and $B$ are {\em consistent in $\pi_i$} if one of them dominates the other in $\pi_i$. We say that $A$ and $B$ are {\em consistent}, if they are consistent in $\pi_i$ for all
$i=1,\ldots,2k-1$. We say that $A$ {\em majority dominates} $B$ if $A$ dominates $B$ in at least $k$
of the linear orders.
Observe that if $A$ majority dominates $B$,
then the bipartite subtournament of $G$ induced by $A$ and $B$ is transitive and all edges go from $A$ to $B$.
Similarly, if $A$ and $B$ are consistent then the bipartite subtournament induced by $A$ and $B$ is transitive, as either $A$ dominates $B$ or vice versa.
Let $d_k(n)$ (resp.\! $d^*_k(n)$) be the largest integer $t$, such that every
$k$-majority tournament with $n$ vertices contains disjoint sets $A$ and $B$ of order $t$
such that $A$ majority dominates $B$ (resp.\! $A$ and $B$ are consistent).
Clearly, we have that $d^*_k(n) \le d_k(n) \le b_k(n)$.

Our next theorem establishes lower and upper bounds for $d^*_k(n)$ and $d_k(n)$.
The lower bound for $d_k(n)$ serves as our lower bound for $b_k(n)$.
Quite surprisingly, the bound we obtain for $d_2(n)$ turns out to be a tight bound also for $b_2(n)$.
One may therefore ask if this phenomenon holds for larger $k$.
We suspect that this is false. Indeed, our upper bounds for $d^*_k(n)/n$ and $d_k(n)/n$
are exponentially small in $k$, while our upper bound for $b_k(n)/n$ is only linear
in $1/k$. We summarize our results on the bipartite parameters in the following theorem.
\begin{theorem}\label{t:2}
	Each of the following limits exists:\\
	$b_k = \lim_{n \rightarrow \infty}n/b_k(n)$,
	$d_k = \lim_{n \rightarrow \infty}n/d_k(n)$,
	$d^*_k = \lim_{n \rightarrow \infty}n/d^*_k(n)$.
	Furthermore,
	\begin{align*}
	\Theta(k) & \le b_k \le \binom{2k}{k},\\
	2^k/e & \le d_k \le \binom{2k}{k},\\
	2^{2k-1}/e & \le d^*_k \le 2^{2k-1},\\
	b_2 & = d_2  = k\,.
	\end{align*}
\end{theorem}

\subsection{Random \texorpdfstring{$k$}{k}-majority tournaments}

As we noted above, every tournament contains a copy of $T_t$ with $t=\log n$, and as is common in many variants
of Ramsey's theorem, random tournaments show that this bound is tight up to a constant factor.
It is thus natural to study the size of the largest transitive tournament in $k$-majority tournaments generated
by picking $2k-1$ random permutations. Besides the intrinsic interest in this question, such random examples might \footnote{We note that
we do use random majority tournaments
as upper-bound constructions in Lemmas \ref{l:strict-domination-upper} and \ref{l:majority-upper}.} very
well give improved upper bounds for $f_k(n)$.
Let us be more formal about this intriguing problem.

Let $\mathcal{G}(n,k)$ denote the probability space of $k$-majority tournaments on vertex set $[n]$
obtained by uniformly choosing at random (with replacement) $2k-1$ linear orders of $[n]$.
Let $X(n,k)$ be the random variable corresponding to the maximum $t$ such that $T_t$ in a subgraph of
$G \sim \mathcal{G}(n,k)$. For example, it is a simple exercise to see that $X(3,2)$ distributes as $2$ with probability $\frac{1}{18}$ and $3$ with probability $\frac{17}{18}$, and it is well known
that $\lim_{k \rightarrow \infty}\Pr[X(3,k)=3]= 0.9123...$ \cite{GK-1968,guilbaud-1952}.
The distribution of $X(n,k)$ becomes involved already for $k=2$ and large $n$. One may consider a possibly easier question of determining (at least asymptotically) $\mathbb{E}[X(n,k)]$.
Specifically, let $r_k$ be the infimum over all $\alpha \in {\mathbb R}$ such that
$\mathbb{E}[X(n,k)] \le n^\alpha$ for all sufficiently large $n$.

While clearly $r_k \le 1$, it is not entirely obvious that $r_k < 1$ nor that $r_k$
decreases as $k$ grows. The answer to the first question is, indeed, positive. As pointed out to us by
Matthew Kwan, it is not difficult to deduce from a result of Cibulka and Kyn{\v{c}}l \cite{CK-2017} as well as from an unpublished result of Lyuben Lychev, that $r_k < 1$ for every $k$. In particular, using
the result from \cite{CK-2017} we can show that:
\begin{proposition}\label{prop:r2}
	$\mathbb{E}[X(n,2)] = O(n^{2/3})$. In particular, $r_2 \le 2/3$.
\end{proposition}
\noindent
Notice that any bound of the form $r_k = o(\ln \ln k/\ln k)$ would improve the upper bound for $f_k(n)$.
We suspect that a much stronger bound holds.
\begin{conjecture}\label{conj:1}
	$r_k = O(1/k)$.
\end{conjecture}
\noindent
Observe that, if true, the above conjecture would match (up to a constant factor) our main result in Theorem \ref{t:1}.
We note that we currently do not even know if $r_k$ decreases with $k$.
It is worth noting that Mossel \cite{mossel-2010} proved that $\lim_{k \rightarrow \infty}\Pr[X(n,k)=n] = e^{-\Theta(n^{5/3})}$ (in particular, the probability that the so-called impartial culture on $n$ vertices is transitive is smaller than any polynomial in $n$, yet significantly larger than the probability that a random tournament is transitive).

\paragraph{Paper organization:} The rest of this paper consists of Section 2 where the bipartite parameters of Theorem \ref{t:2}
are proved. Section 3 contains the proof of Theorem \ref{t:1}, which is based on the bipartite result.
Section 4 contains the proof of Proposition \ref{prop:r2}.

\section{Bipartite transitive subtournaments}

Recall that $d^*_k(n)$ is the largest integer $t$, such that every
$k$-majority tournament with $n$ vertices contains disjoint sets $A$ and $B$ of order $t$
such that $A$ and $B$ are consistent. The following lemma,
whose proof is similar to the proof of Lemma 2.1 of \cite{AM-1986}, gives a lower bound for $d^*_k(n)$.
\begin{lemma}\label{l:strict-domination-lower}
	$d^*_k(n) \ge \lfloor n/2^{2k-1} \rfloor$.
\end{lemma}
\begin{proof}
	By the statement of the lemma we can assume that $n$ is an integer multiple of $2^{2k-1}$.
	Let $G$ be a $k$-majority tournament on vertex set $[n]$ generated by $\pi_1,\ldots,\pi_{2k-1}$.
	For $i=1,\ldots,2k-1$ we construct
	disjoint sets of vertices $A_i$ and $B_i$ with $|A_i|=|B_i| \ge n/2^i$ such that $A_i$ and $B_i$
	are consistent for all $j=1,\ldots,i$.
	
	We proceed inductively by first setting $A_1$ to be the first $n/2$ vertices of $\pi_1$ and setting
	$B_1$ to the remaining vertices of $\pi_1$. Assume that we have already defined $A_i$, $B_i$
	having the claimed properties, in particular $|A_i|=|B_i|=s \ge n/2^i$.
	Let $\sigma$ denote the restriction of $\pi_{i+1}$ to $A_i \cup B_i$,
	so $\sigma$ is a permutation of $A_i \cup B_i$ of order $2s$.
	Let $C$ denote the first $s$ vertices of $\sigma$ and let $D$ be the remaining $s$ vertices of $\sigma$. Now, suppose first that $|A_i \cap C| \ge s/2$. Then we must have that $|B_i \cap D| \ge s/2$
	so we may set $A_{i+1}=A_i \cap C$ and $B_{i+1}=B_i \cap D$.
	Otherwise, we must have $|A_i \cap D| \ge s/2$. Then we must have that $|B_i \cap C| \ge s/2$
	so we may set $A_{i+1}=A_i \cap D$ and $B_{i+1}=B_i \cap C$.
\end{proof}

The lower bound of Lemma \ref{l:strict-domination-lower} is tight up to an absolute constant factor as the following upper bound shows.
\begin{lemma}\label{l:strict-domination-upper}
	$d^*_k(n) \le (1+o(1))en/2^{2k-1} $.
\end{lemma}
\begin{proof}
	We construct a $k$-majority tournament $G$ generated by permutations $\pi_1,\ldots,\pi_{2k-1}$ of $S_n$ having the following property. For any two disjoint sets
	of vertices $A,B$ both of size $t$, $A$ and $B$ are inconsistent in some $\pi_i$. Furthermore, $t \le (1+o(1))en/2^{2k-1}$.
	
	We generate $G$ by picking each $\pi_i$ uniformly at random from $S_n$, where all $2k-1$ choices are independent. For a given pair $\{A,B\}$ of disjoint sets of vertices of order $t$, let $X(A,B)$ be the event that $A$ and $B$ are consistent. We must prove that for the claimed value
	of $t$, with positive probability no event of the form $X(A,B)$ holds. The probability that $A$ and $B$ are consistent in $\pi_i$
	is $2/\binom{2t}{t}$. Hence, $\Pr[X(A,B)]=(2/\binom{2t}{t})^{2k-1}$. Thus,
	summing over all choices of pairs $\{A,B\}$ we have that
	\begin{align*}
	\sum_{\{A,B\}} \Pr[X(A,B)] & \le \frac{1}{2}\binom{n}{t}\binom{n-t}{t} \left(\frac{2}{\binom{2t}{t}} \right)^{2k-1}\\
	& \le \left( \frac{en}{t}\right)^{2t} \left( \frac{1+o(1)}{4}\right)^{t(2k-1)} < 1
	\end{align*}
	where the last strict inequality holds for all $t \ge (1+o(1))en/2^{2k-1}$. Thus, by the union bound,
	with positive probability no event of the form $X(A,B)$ holds.
\end{proof}

We next turn to $d_k(n)$ which, recall, is the largest integer $t$, such that every
$k$-majority tournament with $n$ vertices contains disjoint sets $A$ and $B$ of order $t$
such that $A$ majority dominates $B$.
Our lower bound for $d_k(n)$ is greater than our lower bound for $d^*_k(n)$ by a factor of $\Theta(\sqrt{k})$.
\begin{lemma}\label{l:majority-lower}
	Let $G$ be a $k$-majority tournament with $n$ vertices. Then $V(G)$ can be partitioned into
	$2\binom{2k-1}{k}=\binom{2k}{k}$ disjoint parts $\{A_s,B_s\}$ for $s=1,\ldots,\binom{2k-1}{k}$ such that $A_s$ majority dominates $B_s$ and
	$||A_s|-|B_s|| \le 1$. In particular,
	$$
	b_k(n) \ge d_k(n) \ge \left\lfloor \frac{n}{\binom{2k}{k}}\right\rfloor \ge
	\left\lfloor \sqrt{\pi(k-1)/4} \cdot \frac{n}{2^{2k-1}} \right\rfloor \;.
	$$
\end{lemma}
\begin{proof}
	For a binary vector $t$, let $t^i$ denote its prefix of length $i$, let $t(i)$ denote its $i$'th coordinate, and let $\overline{t}$ denote its antipode.
	
	Suppose $\pi_1,\ldots,\pi_{2k-1}$ generate $G$. We associate with each $v \in V(G)$ a
	{\em type} which is a binary vector $t_v$ of length $2k-1$ and which is defined by the following procedure. We will require that for each binary vector $s$ of length $i$, the cardinality of the set of all vertices $v$ with $t^i_v=s$ differs from the cardinality of the set of all vertices with
	$t^i_v=\overline{s}$ by at most $1$. We now define the coordinates of $t_v$ inductively.
	All the vertices $v$ in the first $\lfloor n/2 \rfloor$ positions of $\pi_1$ will
	have $t_v(1)=0$ and the remaining vertices will have $t_v(1)=1$.
	Observe that the requirement on cardinalities holds for length $i=1$.
	Assume that we have already defined the first $i$ coordinates of the types of all vertices.
	We show how to define the $(i+1)$'th coordinate.
	For a binary vector $s$ of length $i$ let $A_s$ be the set of all vertices $v$ with
	$t^i(v)=s$ and let $B_s$ be the set of all vertices $v$ with $t^i(v)=\overline{s}$.
	By the inductive assumption, $||A_s| - |B_s|| \le 1$.
	Let $\sigma$ denote the restriction of $\pi_{i+1}$ to $A_s \cup B_s$.
	Let $C$ be the vertices in the first $|A_s|$ positions of $\sigma$ and let $D$ be the vertices in the remaining $|B_s|$ positions. All $v \in C$ will have $t_v(i+1)=0$ and all $v \in D$ will have $t_v(i+1)=1$. This defines the $(i+1)$'th coordinate and observe that the requirement on cardinalities also holds for length $i+1$.
	
	Having defined types, we see that $V(G)$ can be partitioned into $2^{2k-1}$ parts according to types.
	We can, however, make our partition coarser. We define a set $Q$ of binary vectors inductively.
	The members of $Q$ will have lengths from $k$ to $2k-1$. The only vectors of length $k$ in $Q$
	are the all zero vector and the all one vector.
	Assume we have already defined all the vectors of length $k+i-1$ in $Q$. A vector $s$ of length $k+i$ which has at least $k$ one's or which has at least $k$ zeros and whose prefix $s^{k+i-1}$ is not yet in $Q$ will belong to $Q$. Observe that the number of vectors of $Q$ of length $k+i$ is $2\binom{k-1+i}{i}$ and hence $|Q|=\binom{2k}{k}$.
	Furthermore, if $s \in Q$ then $\overline{s} \in Q$.
	Also observe that each binary vector of length $2k-1$ has a prefix in $Q$.
	
	Our coarser partition will be according to the elements of $Q$. For a vector $s \in Q$ with at least
	$k$ zeros, let $A_s$ be the set of all vertices whose type has $s$ as its prefix.
	Similarly define $B_s$ as the set of all vertices whose type has $\overline{s}$ as its prefix.
	Observe that by our construction of types, $||A_s|-|B_s|| \le 1$ and clearly $A_s$ majority dominates $B_s$.
\end{proof}
	
An upper bound for $d_k(n)$ is proved by the exact same randomized construction as in Lemma  \ref{l:strict-domination-upper}, and hence the proof is analogous. We thus obtain the following lemma whose proof is omitted.
\begin{lemma}\label{l:majority-upper}
	$d_k(n) \le (1+o(1))en/2^{k}$. \qed
\end{lemma}

Interestingly, the lower bound for $b_2(n)$ given by Lemma \ref{l:majority-lower} is asymptotically tight.
Before we prove that, we recall the lexicographical graph product (in our case - tournament product) which will also be useful in subsequent results. For two tournaments $G_1$ and $G_2$, let
$H = G_1 \circ G_2$ be the tournament on vertex set $V(G_1) \times V(G_2)$ where $(u_1,v_1),(u_2,v_2)$
is an edge if $v_1 \neq v_2$ and $(v_1,v_2) \in E(G_2)$ or if $v_1=v_2$ and $(u_1,u_2) \in E(G_1)$.
(notice that the lexicographical product is, in general, non-commutative).
For a tournament $G$ we denote the $r$'th power of $G$ by $G^r=G^{r-1} \circ G$. It is not difficult to observe that if $G$ is a $k$-majority tournament, so is $G^r$ (see \cite{MSW-2011} Proposition 3.2).

\begin{lemma}\label{l:b2n}
	$b_2(n) = (1+o(1))n/6$.
\end{lemma}
\begin{proof}
	By Lemma \ref{l:majority-lower}, $b_2(n) \ge \lfloor n/6 \rfloor$. For the upper bound, we prove that for every $n$ which is a power of $7$ it holds that $b_2(n)=(n-1)/6$. The result then follows from the fact that $b_2(n)/n$ has a limit.
	
	Let $P$ denote the Paley tournament on vertex set $\{1,\ldots,7\}$ which is also the unique (regular) tournament on seven vertices with no $T_4$ as a subgraph. It is well-known that $P$, and hence $P^r$, is a $2$-majority tournament (see \cite{ABKKW-2006}).
	We prove that the largest $t$ for which
	$P^r$ has a $T_{t,t}$ is $t=(7^r-1)/6$.	
	Observe that this holds for $r=1$ since a tournament has a $T_{2,2}$ if and only if it has a $T_4$ and $P$ has no $T_4$.
	Now consider $P^r=P^{r-1}\circ P$ which has $7^r$ vertices. Thus, the vertex set of $P^r$
	has seven parts $V_1,\ldots,V_7$ where the bipartite subtournament induced by the pair $V_i,V_j$
	is a $T_{7^{r-1},7^{r-1}}$ where all edges go from $V_i$ to $V_j$ if and only if $(i,j) \in E(P)$.
	Furthermore, the subtournament induced by each $V_i$ is $P^{r-1}$.
	By induction, the largest $t$ for which
	$P^{r-1}$ has a $T_{t,t}$ is $t=(7^{r-1}-1)/6$.
	
	Now, consider some $T_{t,t}$ in $P^r$ and let its parts be $A,B$.
	Let $X=\{i| A \cap V_i \neq \emptyset\}$ and $Y=\{i| B \cap V_i \neq \emptyset\}$.
	If $|X|=1$  then $t=|A| \le 7^{r-1} \le (7^r-1)/6$ and we are done.
	Similarly, if $|Y|=1$  then $t=|B| 7^{r-1} \le (7^r-1)/6$ and we are done.
	So we may assume that $|Y| \ge |X| \ge 2$. Without loss of generality, assume all edges go from $A$ to $B$.
	Suppose $u,v \in X$. Since $P$ has no $T_4$, any two vertices $u,v$ of $P$ have at most one
	common out-neighbor, say $w$. Hence, $Y \subseteq \{w,u,v\}$. However, we cannot have that both
	$u,v$ are in $Y$ since otherwise not all edges would go from $A$ to $B$. Hence,
	we must have that $|Y|=2$ and hence $X=\{u,v\}$ and either
	$Y = \{w,u\}$ or $Y =\{w,v\}$.
	If $Y = \{w,u\}$ then by the induction hypothesis, at least one of $A$ or $B$ cannot contain more than $(7^{r-1}-1)/6$ vertices of $u$.
	Similarly, if $Y=  \{w,v\}$ then at least one of $A$ or $B$ cannot contain more than
	$(7^{r-1}-1)/6$ vertices of $v$.
	In any case, this implies that
	$$
	t \le 7^{r-1} + (7^{r-1}-1)/6 = (7^r-1)/6 \;.
	$$
\end{proof}

Our next result is an upper bound for $b_k(n)$. While our upper and lower bounds for $d^*_k(n)$ and $d_k(n)$ are all exponentially small in $k$, our upper bound for $b_k(n)$ is only linear in $k$.
\begin{lemma}\label{l:bk-upper}
	There is an absolute constant $c$ such that
	$b_k(n) \le cn/k$.
\end{lemma}
\begin{proof}
	A result of Erd\H{o}s and Moser \cite{EM-1964} asserts that the largest $q$ such that
	every tournament with $q$ vertices is a $k$-majority tournament, satisfies $q=\Omega(k \log k)$.
	In fact, $q=\Theta(k \log k)$ as proved by Stearns \cite{stearns-1959}.
	Consider a random tournament $G$ on $q$ vertices.
	It is a standard probabilistic argument to prove that $G$ has no $T_{b+1,b+1}$
	for $b=O(\log q)=O(\log k)$ so we fix such a tournament $G$.
	Observe that $(q-1)/b \ge ck$ for some absolute positive constant $c$.
	
	We prove that for every $n$ which is a power of $q$ it holds that $b_k(n) \le b(n-1)/(q-1)$.
	The result stated in the theorem then follows from the fact that $b_k(n)/n$ has a limit.
	The proof generalizes the arguments
	in Lemma \ref{l:b2n}. We prove that the largest $t$ for which
	$G^r$ has a $T_{t,t}$ is $t=b(q^r-1)/(q-1)$.	
	Observe that this holds for $r=1$ as $G$ has no $T_{b+1,b+1}$.
	Now consider $G^r=G^{r-1}\circ G$ which has $q^r$ vertices.
	Thus, the vertex set of $G^r$
	has $q$ parts $V_1,\ldots,V_q$ where the bipartite subtournament induced by the pair $V_i,V_j$
	is a $T_{q^{r-1},q^{r-1}}$ where all edges go from $V_i$ to $V_j$ if and only if $(i,j) \in E(G)$.
	Furthermore, the subtournament induced by each $V_i$ is $G^{r-1}$.
	By induction, the largest $t$ for which
	$G^{r-1}$ has a $T_{t,t}$ is $t=b(q^{r-1}-1)/(q-1)$.
	
	Now, consider some $T_{t,t}$ in $G^r$ and let its parts be $A,B$.
	Let $X=\{i| A \cap V_i \neq \emptyset\}$ and $Y=\{i| B \cap V_i \neq \emptyset\}$.
	If $|X| \le b$, then $t=|A| \le bq^{r-1} \le b(q^r-1)/(q-1)$ and we are done.
	Similarly, if $|Y| \le b$, then $t=|B| \le bq^{r-1} \le b(q^r-1)/(q-1)$  and we are done.
	So we may assume that $|Y| \ge |X| \ge b+1$. Without loss of generality, assume all edges go from $A$ to $B$ (otherwise we can consider the transpose of $G$).
	Suppose $u_1,\ldots,u_{b+1} \in X$. Since $G$ has no $T_{b+1,b+1}$,
	the vertices $u_1,\ldots,u_{b+1}$ have at most $b$
	common out-neighbors in $G$, say these are $w_1,\ldots,w_b$.
	Hence, $Y \subseteq \{w_1,\ldots,w_b,u_1,\ldots,u_{b+1}\}$. However, we cannot have that both
	$u_i,u_j$ are in $Y$ for distinct $i,j$ since otherwise not all edges would go from $A$ to $B$.
	Hence, we must have that $|Y|=b+1$ and hence $X=\{u_1,\ldots,u_{b+1}\}$ and
	$Y = \{w_1,\ldots,w_b,u_j\}$ for some $j$.
	By the induction hypothesis, at least one of $A$ or $B$ cannot contain more than $b(q^{r-1}-1)/(q-1)$ vertices of $u_j$. This implies that
	$$
	t \le bq^{r-1} + b(q^{r-1}-1)/(q-1) = b(q^r-1)/(q-1) \;.
	$$
\end{proof}	
We note that in Lemma \ref{l:bk-upper} it is important to use the $\Omega(k \log k)$
bound of Erd\H{o}s and Moser on the largest $q$ such that
every tournament with $q$ vertices is a $k$-majority tournament.
Using weaker bounds, such as $q=\Omega(k)$, which are much easier to prove, will give only sublinear bounds in $1/k$ for $b_k(n)$.

\begin{lemma}\label{l:limit-biparite}
	All the limits stated in Theorem \ref{t:2} exist.
\end{lemma}
\begin{proof}
	We prove that $\lim_{n \rightarrow \infty}b_k(n)/n$ exists. The proofs of the other limits is analogous, and simpler.
	Let	$s_k = \liminf_{n \rightarrow \infty} b_k(n)/n$.
	Let $\epsilon > 0$. We must prove that for all $n$ sufficiently large,
	$b_k(n) \le (\epsilon+s_k)n$ as this will prove that
	$\limsup_{n \rightarrow \infty} b_k(n)/n \le s_k+\epsilon$ and hence the limit exists.
	
	So, let $\epsilon > 0$ be given. By the definition of $\liminf$, there exist infinitely many
	positive integers
	$q$ where for each such $q$ there is $k$-majority tournament $G$ with $q$ vertices where the largest $T_{b,b}$ of $G$ has $b \le (s_k+\epsilon/2)q$.
	Hence, we may additionally require that $q \ge 2(s_k+\epsilon)/\epsilon$.
	
	Now suppose than $n > q$.
	Let $r$ be the smallest integer such that $q^r \ge n$, hence, $q^{r-1} \le n$.
	Consider the tournament $G^r$. It is a $k$-majority tournament and, as proved in Lemma \ref{l:bk-upper} the largest transitive balanced bipartite tournament
	of $G^r$ has at most $b(q^{r-1}-1)/(q-1)$ vertices in each part. Hence,
	$$
	b_k(n) \le b_k(q^r) \le b\frac{q^{r-1}-1}{q-1} \le \frac{b}{q-1}n \le \frac{(s_k+\epsilon/2)q}{q-1}n \le
	(\epsilon+s_k)n\;.
	$$
\end{proof}		

Finally, observe that all the lemmas proved in this section together prove Theorem \ref{t:2}. \qed

\section{Transitive subtournaments}

Our lower bound for $f_k(n)$ is established by the following lemma.

\begin{lemma}\label{l:fkn-lower}
$f_k(n) \ge n^{1/(2k-\frac{1}{2}\log_2 k)}$.
\end{lemma}
\begin{proof}
Our proof is based on recursively applying the lower bound for the bipartite parameter $b_k(n)$.
We note that a similar approach is used in \cite{APPRS-2005}.
The proof is by induction on $n$ where the base case clearly holds for $n = 1$.
Assume the statement holds for values smaller than $n$
and let $G$ be a $k$-majority tournament with $n$ vertices.
By Lemma \ref{l:majority-lower}, $G$ contains a $K_{t,t}$ with
$t \ge \lfloor \sqrt{\pi(k-1)/4} \cdot \frac{n}{2^{2k-1}} \rfloor$,
so in particular, $t \ge \sqrt{k}n/2^{2k}$ for all $n \ge 2$ (since we can always assume that $t \ge 1$).
Let the parts of such $K_{t,t}$ be $A,B$.
By the induction hypothesis, each of $G[A]$ and $G[B]$ has a $T_s$ with $s \ge t^{1/(2k-\frac{1}{2}\log_2 k)}$. The union of these two $T_s$ subtournaments is a $T_{2s}$ in $G$ and it holds that
$$
2s \ge 2t^{1/(2k-\frac{1}{2}\log_2 k)} \ge 2(\sqrt{k}n/2^{2k})^{1/(2k-\frac{1}{2}\log_2 k)} \ge  n^{1/(2k-\frac{1}{2}\log_2 k)}\;.
$$
\end{proof}

\begin{lemma}\label{l:limit-fkn}
$\lim_{n \rightarrow \infty} \log_n f_k(n)$ exists.
\end{lemma}
\begin{proof}
Let	$s_k = \liminf_{n \rightarrow \infty} \log_n f_k(n)$.
Let $\epsilon > 0$. We must prove that for all $n$ sufficiently large,
$f_k(n) \le n^{\epsilon+s_k}$ as this will prove that
$\limsup_{n \rightarrow \infty} \log_n f_k(n) \le s_k+\epsilon$ and hence the limit exists.

So, let $\epsilon > 0$ be given. By the definition of $\liminf$, there exists a positive integer $q$ and a $k$-majority tournament $G$ with $q$ vertices where the largest $T_t$ of $G$ has $t \le q^{s_k+\epsilon/2}$.

Let $N = q^{1+2s_k/\epsilon}$ and suppose that $n > N$.
Let $r$ be the smallest integer such that $q^r \ge n$, hence, $q^{r-1} \le n$.
Consider the tournament $G^r$. It is a $k$-majority tournament and the largest transitive subtournament
of $G^r$ has precisely $t^r$ vertices. Hence,
$$
f_k(n) \le f_k(q^r) \le t^r \le (q^r)^{s_k+\epsilon/2} \le (qn)^{s_k+\epsilon/2} \le n^{s_k+\epsilon}
$$
where in the last inequality we have used that $n \ge q^{1+2s_k/\epsilon}$.
\end{proof}

Lemmas \ref{l:fkn-lower} and \ref{l:limit-fkn} together imply Theorem \ref{t:1}. \qed

\section{Transitive sets in random \texorpdfstring{$2$}{2}-majority tournaments}
In this section we prove Proposition \ref{prop:r2}. Our main tool is Theorem 1.6 of \cite{CK-2017} for which we need to introduce some notation and definitions.
A {\em $d$-dimensional binary matrix of order $n$} (hereafter $(d,n)$-matrix) is a function from $[n]^d$ to $\{0,1\}$.
A {\em $(d,n)$-permutation matrix} is a $(d,n)$-matrix such that for every $j \in [d]$ and every $i \in [n]$, all but one of the entries of the $(d-1,n)$-matrix obtained by restricting the $j$'th dimension to $i$ are zero.
Notice that usual $n$-permutation matrices are just $(2,n)$-permutation matrices.
For a $(d,n)$-matrix $A$ and a $(d,k)$-permutation matrix $P$, we say that
$A$ {\em contains the pattern $P$}
if there exist $d$ strictly increasing $k$-sequences contained in $[n]$,
$s_i=(s_{i,1},\ldots,s_{i,k})$ for $i \in  [d]$ such that if
$A(s_{1,j_1},s_{2,j_2},\ldots,s_{d,j_d})=1$, then $P(j_1,\ldots,j_d)=1$.
Otherwise, we say that $A$ {\em avoids} the pattern $P$.
For a fixed $(d,k)$-permutation matrix $P$, let $S_P(n)$ be the number of $(d,n)$-permutation matrices avoiding $P$. Extending the celebrated theorem of Marcus and Tardos \cite{MT-2004} (which, in turn, implies the Stanley-Wilf Conjecture stating that $S_P(n) \le c_P^n$ for a $(2,k)$-permutation matrix $P$) to higher dimensions, Cibulka and Kyn{\v{c}}l \cite{CK-2017} proved the following result:
\begin{lemma}\label{l:ck}[\cite{CK-2017}, Theorem 1.6]
	For every $d, k \ge 2$ and every $(d,k)$-permutation matrix $P$, we have
	$$
	n^{-O_d(k)}\left(\Omega_d\left(k^{1 /\left(2^{d-1}(d-1)\right)}\right)\right)^n \cdot(n!)^{d-1-1 /(d-1)} \le S_P(n) \le \left(2^{O_d(k)}\right)^n \cdot(n!)^{d-1-1 /(d-1)}\;.
	$$
\end{lemma}
For $G \sim \mathcal{G}(n,2)$, let $p_n$ be the probability that $G=T_n$ (i.e., that $G$ is a transitive tournament). As relabeling vertices does not affect $p_n$, each element of
$\mathcal{G}(n,2)$ corresponds to a pair of $n$-permutations $(\pi_1,\pi_2)$ and
$G$ is obtained by selecting one of the $(n!)^2$ possible pairs uniformly at random and
constructing the $2$-majority tournament from the triple ${(\rm id}, \pi_1, \pi_2)$.
So, if $|F(n)|$ denotes the set of pairs $(\pi_1,\pi_2)$ such that the corresponding $G$ is transitive, we have that $p_n = |F(n)|/(n!)^2$. Now suppose that for $a < b < c$, the (single line notation of) $\pi_1$ contains $b,c,a$ as a substring and that $\pi_2$ contains $c,a,b$ as a substring. Then $G$ contains a directed triangle on the vertices $a,b,c$ and hence $(\pi_1,\pi_2) \notin F(n)$.
Let $F^*(n)$ be the set of pairs $(\pi_1,\pi_2)$ such that for any $a < b < c$, $\pi_1$ does not contain $b,c,a$ as a substring or
else $\pi_2$ does not contain $c,a,b$ as a substring. We therefore have $F(n) \subseteq F^*(n)$.
Consider the $(3,3)$-permutation matrix $P$ with $P(1,2,3)=1$, $P(2,3,1)=1$ and $P(3,1,2)=1$
and consider some $(3,n)$-permutation matrix $A$.
Note that $A$ bijectively maps to a pair $(\pi_1,\pi_2)$ by $A[i,j,k]=1$ meaning that $i$ appears in the $j$'th position in $\pi_1$ and the $k$'th position in $\pi_2$. Furthermore, $A$ avoids the pattern
$P$ if and only if the corresponding $(\pi_1,\pi_2) \in F^*(n)$, so $S_P(n)=|F^*(n)|$.
It follows from Lemma \ref{l:ck} that $|F(n)| \le C^n (n!)^{1.5}$ for some constant $C$,
whence $p_n \le C^n(n!)^{-1/2}$.

Using this upper bound for $p_n$, we may now complete the proof of Proposition \ref{prop:r2}.
Let $m$ be a parameter (depending on $n$) chosen appropriately later.
Consider $G \sim \mathcal{G}(n,2)$ where $G$ is generated from the permutations
${(\rm id}, \pi_1, \pi_2)$. Let $A \subset [n]$ with $|A|=m$ and consider the subtournament $G[A]$
noticing that $G[A] \sim  \mathcal{G}(m,2)$. Thus, the probability that $G[A]$ is transitive is
$p_m \le C^m(m!)^{-1/2}$ where we may clearly assume that $C \ge 1$. By the union bound,
$$
\Pr[X(n,2) > m] \le  \binom{n}{m}C^m(m!)^{-1/2}\;.
$$
Using $m=Cen^{2/3}$ we have
\begin{align*}
\Pr[X(n,2) > Cen^{2/3}] & \le  \binom{n}{Cen^{2/3}}C^{Cen^{2/3}}((Cen^{2/3})!)^{-1/2}\,\\
& \le n^{\frac{1}{3}Cen^{2/3}}(Cn^{2/3})^{-Cen^{2/3}/2} \\
& = (C^{-Ce/2})^{n^{2/3}}
\end{align*}
So, the probability that $X(n,2) > Cen^{2/3}$  is exponentially small in $n^{2/3}$, implying that
$\mathbb{E}[X(n,2)] = O(n^{2/3})$, as required. \qed

\section*{Acknowledgment}

The authors are grateful to Noga Alon for useful discussions and to Matthew Kwan for pointing out that the result from \cite{CK-2017} can be used to obtain Proposition \ref{prop:r2}.

\bibliographystyle{abbrv}

\bibliography{references}

\end{document}